\documentclass[15pt, twoside]{article}
\usepackage{amssymb}
\usepackage[all]{xy}
\usepackage{amsthm,amsmath,color}
\usepackage{pgfmath}
\usepackage{tikz-cd}
\usetikzlibrary{arrows, matrix}
\usepackage{zref-abspage}
\usepackage{perpage}
\usepackage[top=2.3cm,right=3.5cm,bottom=2.5cm,left=2.5cm]{geometry}
\usepackage{graphicx}
\usepackage{fancyhdr}
\usepackage{makeidx}
\usepackage{ifpdf}
\usepackage[colorlinks=true, linkcolor=blue, linktoc=all, citecolor=magenta, filecolor=cyan, urlcolor=cyan, linkbordercolor={1 0 0}, citebordercolor={0 1 0}, urlbordercolor={0 1 1}]{hyperref}
\newtheorem{theorem}{Theorem}[section]
\newtheorem{definition}{Definition}[section]
\newtheorem{proposition}{Proposition}[section]

\newtheorem{corollary}{Corollary}[section]
\newtheorem{example}{Example}[section]
\date{}
\begin{document}
\fancyhead{}
\fancyfoot{}
\chead[M. J. Afshari, S. Varsaie]{Homotopy Classification of Super Vector Bundles and Universality}
\lhead[\thepage]{}
\rhead[]{\thepage}
\pagestyle{fancy}
\author{Mohammad Javad Afshari\footnote{\textit{E-mail addresses:} \href{mailto: afshari.mj@iasbs.ac.ir}%
{afshari.mj@iasbs.ac.ir} (M. J. Afshari), \href{mailto: varsaie@iasbs.ac.ir}%
{varsaie@iasbs.ac.ir} (S. Varsaie).} , Saad Varsaie
\\
\small
\textit{Department of Mathematics,}
\\
\small
\textit{Institute for Advanced Studies in Basic Sciences (IASBS),}
\\
\small
\textit{Zanjan 45137-66731, Iran}}                            
\title{Homotopy Classification of Super Vector Bundles and Universality}
\renewcommand{\baselinestretch}{1.5}
\setlength{\baselineskip}{1.5\baselineskip}
\setcounter{page}{1}
\maketitle
\begin{abstract}
This study first provides a brief overview of the structure of typical Grassmann manifolds. Then a new type of supergrassmannians is construced using an odd involution in a super ringed space and by gluing superdomains together. Next, constructing a Gauss morphism of a
super vector bundle, some properties of this morphism is discussed. By this, we generalize one of
the main theorems of homotopy classification for vector bundles in supergeometry. Afterwards, a similar structure is introduced in the state of infinite-dimensional. Here our tools mainly include multilinear algebra of Grassmann algebras, the direct limit of the base spaces and the inverse limit of the structure sheaf of ringed spaces. We show that the resulting super vector bundle is a universal member of its category.
\\
\textbf{AMS 2020 subject classiﬁcations:} Primary, 58A50, 55R15; Secondary, 54B40, 55P10.
\end{abstract}
\section*{Introduction}	
This paper aims at extending a homotopy classification for super vector bundles. In the category of vector bundles,  it is shown that canonical vector bundles $\gamma^n_k$ on Grassmannians $Gr(n,k)$ are universal. Equivalently, associated to each vector bundle $\cal{E}$ on $ M $, up to homotopy, there exists a unique map $f:M\to Gr(n,k)$, for sufficiently large $n$, such that $\cal{E}$ is isomorphic with the induced bundle of $\gamma^n_k$ under $f$. 
\\
Totally a universal vector bundle, $ \gamma_k $, is a canonical vector bundle on an infinite dimensional Grassmannian, $ Gr_k $. It is shown that any vector bundle of rank $ k $  on a compact manifold, $ B $, is isomorphic to the pullback of a canonical vector bundle, $ \gamma_k $, along a suitable mapping such as  $B\rightarrow Gr_k $, which is proven to be unique up to homotopic approximation. As a result, there is an isomorphism,
\begin{equation*}
Vect_k(B)\cong [B,Gr_k],
\end{equation*}
between the set of isomorphism classes of rank $ k $ with base $ B $ and the set of homotopic classes of mappings  $B\rightarrow Gr_k$ . In this way, the category of vector bundles can be classified \cite{husemuller}. For this reason, the infinite-dimensional Grassmann manifold is called the classification space, and the canonical vector bundle on it is called the universal vector bundle. The class of these maps corresponds to the invariants called cohomology classes on $ B $, known as characteristic classes. All characteristic classes are derived from these maps and universal Chern classes which are characteristic classes of  Grassmann spaces $ Gr_k $ \cite{Bartocci-B-H}. In fact, Chern classes of a vector bundle may be described as the pullback of Chern classes of the universal bundle. In general, this is true for vector bundles with paracompact base spaces and so for any vector bundle whose base space is a manifold (because the compact local and second countable spaces are paracompact). In \cite{Roshandel}, some cohomology elements called $\nu$-classes, as a supergeneralization of universal Chern classes, are introduced for canonical super vector bundles over
$\nu$-Grassmannians. Our result (corollary \ref{sigmat}) in this work make it possible to extend the concept of Chern classes for all super vector bundles. In physics, Chern classes are related to special sort of quantum numbers called topological charges \cite{David}.
Therefore, the appropriate generalization of these classes in supergeometry can be a mathematical framework for describing similar charges in supersymmetry.\\
The supergrassmannians introduced in \cite{Manin} and the Grassmannians, in some sense, are homotopy equivalent (cf. subsection \ref{grassmannian}). So, the cohomology group associated to supergrassmannian is equal to that of Grassmannian. In other words, the former group contains no information about the superstructure. Hence, from classifying space viewpoint, supergrassmannians are not good generalization of Grassmannians. Therefore, new generalizations entitled $\nu$-Grassmannians have been invented.
\\
The study is organized in three sections. In the preliminaries section, following \cite{Bahadory}, we introduce $\nu$-Grassmannians denoted by $_{\nu}Gr(m|n)$, as a new supergeneralization of Grassmannians. In the second section, we show the existence of $\Gamma$, a canonical super vector bundle over $_{\nu}Gr(m|n)$. After introducing Gauss morphisms for super vector bundles, some properties of this morphisms are studied. Then, we extend one of the main theorems on homotopy classification for vetor bundles to supergeometry. In the last section, we study infinite dimensional Grassmann manifolds as ringed spaces. For this end, we obtain the Grassmannian base space using the direct limit and its sheaf structure using the inverse limit. Then, the sheaf structure of infinite dimensional $\nu$-Grassmannians (ringed spaces) is defined by the use of the direct limit \cite{Dugun}. After construcing the canonical super vector bundle over  $ _\nu Gr_{k|l}^{\infty} $, we show that this super bundle is universal.
\\
There are different approaches to generalize  Chern classes in supergeometry, such as homotopy or analytic approach. In this paper our approach is homotopic. Although there are not many articles with homotopy approach, but one may refer to \cite{V-manin} as a good example for such papers. Nevertheless, much more efforts have been made for generalizing Chern classes in supergeometry by analytic approach. One may refer to \cite{Bartocci-B}, \cite{Bartocci-B-H}, \cite{B-H}, \cite{Landi}, \cite{Manin-P}, \cite{V-Manin-P}. But, in all these works, the classes obtained in this way are nothing but the Chern classes of the reduced vector bundle(s) and they do not have any information about the superstructure.  
\section{Preliminaries}
In this section, first, we recall some basic definitions of supergeometry. Then, we introduce a supergeneralization of Grassmannian called $\nu$-Grassmannian.
\subsection{Supermanifolds}
A \textit{super ringed space} is a pair $(X, \mathcal{O})$ where $X$ is a topological space and $\mathcal{O}$ is a sheaf of commutative $\mathbb{Z}_{2}$-graded rings with units on $X$. Let $\mathcal{O}(U)= \mathcal{O}^{ev}(U) \oplus \mathcal{O}^{odd}(U)$ for any open subset $U$ of $X$. An element $a$ of $\mathcal{O}(U)$ is called a homogeneous element of parity $p(a)=0$ if $a \in \mathcal{O}^{ev}(U)$  and it is a homogeneous element of parity $p(a)=1$ if $a \in \mathcal{O}^{odd}(U)$. A morphism between two super ringed spaces $(X, \mathcal{O}_{X})$ and $(Y, \mathcal{O}_{Y})$ is a pair $(\widetilde{\psi}, \psi^{*})$ such that $\widetilde{\psi}: X \longrightarrow Y$ is a continuous map and $\psi^{*}: \mathcal{O}_{Y} \longrightarrow \widetilde{\psi}_{*}(\mathcal{O}_{X})$ is a homomorphism between the sheaves of $\mathbb{Z}_{2}$-graded rings.
\\
\begin{definition}
A \textit{superdomain} is a super ringed space $\mathbb{R}^{p|q}:=(\mathbb{R}^{p}, \mathcal{O})$ where
\begin{equation*}
\mathcal{O}= \mathbf{C}^{\infty}_{\mathbb{R}^{p}} \otimes_{\mathbb{R}} \wedge \mathbb{R}^{q}, \qquad p, q \in \mathbb{N}.
\end{equation*}
\end{definition}
By $\mathbf{C}^{\infty}_{\mathbb{R}^{p}}$ we mean the sheaf of smooth functions on $\mathbb{R}^{p}$. 
\\
A super ringed space which is locally isomorphic to $\mathbb{R}^{p|q}$ is called a \textit{supermanifold} of dimension $p|q$. Note that a morphism $(\widetilde{\psi}, \psi^{*})$ between two supermanifolds $(X, \mathcal{O}_{X})$ and $(Y, \mathcal{O}_{Y})$ is just a morphism between the super ringed spaces such that for any $x \in X$, $\psi^*: \mathcal{O}_{Y, \widetilde{\psi}(x)} \longrightarrow \widetilde{\psi}_{*}(\mathcal{O}_{X, x})$ is local, i.e., $\psi^{*}(m_{\widetilde{\psi}(x)}) \subseteq m_{x}$, where $m_x$ is the unique maximal ideal in $\mathcal{O}_{X, x}$.
 \subsection{$\nu$-Grassmannian}\label{grassmannian}
Supergrassmannians are not good generalization of Grassmannians. Indeed these two, in some sense, are homotopy equivalent. This equivalency may be shown easily in the case of projective superspaces.
\\
To this end, let $\mathbb{P}^{m|n} = (\mathbb{RP}^{m}, \mathcal{O}_{\mathbb{P}^{m|n}})$ be the real projective superspace.
By a \textit{deformation retraction} from $\mathbb{P}^{m|n}$ to $\mathbb{P}^{m}$, we mean that there is a morphism $\mathbb{P}^{m|n} \times \mathbb{R}^{1|0} \longrightarrow\mathbb{P}^{m|n}$ such that the following diagrams commute:

\begin{displaymath}
\xymatrix{
	{\mathbb{P}^{m|n} \times \mathbb{R}^{1|0}} {\ar[r]^{\quad H}}  &
	\mathbb{P}^{m|n}  \\
	\mathbb{P}^{m|n} \ar[u]^{j_0} \ar[ur]_{id}  }
,\qquad
\xymatrix{
	{\mathbb{P}^{m|n} \times \mathbb{R}^{1|0}} {\ar[r]^{\quad H}}  &
	\mathbb{P}^{m|n}  \\
	\mathbb{P}^{m|n} \ar[u]^{j_1} \ar[ur]_{j \circ r}  },
\end{displaymath}
where $ j_0 $ and $ j_1 $ are morphisms from $ \mathbb{P}^{m|n} $ to $ \mathbb{P}^{m|n} \times \mathbb{R}^{1|0} $ corresponding to pairs $ (id,0) $ and $ (id,1) $ w.r.t. the bijection $ Hom(\mathbb{P}^{m|n}  , \mathbb{P}^{m|n}\times \mathbb{R}^{1|0}) \cong Hom(\mathbb{P}^{m|n}  , \mathbb{P}^{m|n}) \times Hom(\mathbb{P}^{m|n}  , \mathbb{R}^{1|0}) $ as a result of 
categorical product in category of supermanifolds.
In addition $ j:\mathbb{P}^{m}  \longrightarrow\mathbb{P}^{m|n} $ is the embedding of reduced manifold $ \mathbb{P}^{m} $ into the supermanifold $ \mathbb{P}^{m|n} $. For more details on catagorical product and reduced manifolds see \cite{Varadarajan}, pages 137 and 133 respectively. $ r:\mathbb{P}^{m|n} \longrightarrow\mathbb{P}^{m} $ is a morphism such that $ \tilde{r}=id_{\mathbb{P}^{m}} $ and $ r^*:\mathcal{O}_{\mathbb{P}^{m}} \longrightarrow \mathcal{O}_{\mathbb{P}^{m|n}} $ is a sheaf homomorphism induced by the local homomorphisms $ \mathcal{O}_{\mathbb{P}^{m}}(U_\alpha) \stackrel{r_{\alpha}^*}{\longrightarrow} \mathcal{O}_{\mathbb{P}^{m|n}}(U_\alpha) $ with $ x^{\alpha}_i \longmapsto x^{\alpha}_i $ for an open covering $\{U_{\alpha}\}_{\alpha}$ s.t., for each $ \alpha, $ $(U_{\alpha}, \mathcal{O}_{\mathbb{P}^{m|n}}|_{U_{\alpha}})$ is isomorphic to a superdomain and $ (x_i^{\alpha})_i $ are any local coordinates  on $ U_{\alpha} $.
\begin{proposition}
	There is a deformation retraction from $\mathcal{O}_{\mathbb{P}^{m|n}}$ to $\mathcal{O}_{\mathbb{P}^{m}}$.
\end{proposition}
\begin{proof}
Let $ \{U_{\alpha}\} $ be a covering defined as above and set $ U_{\alpha}^{m|n}=(U_{\alpha}, \mathcal{O}_{\mathbb{P}^{m|n}}|_{U_{\alpha}}) $. Now consider a morphism $ H:\mathbb{P}^{m|n} \times \mathbb{R}^{1|0} \longrightarrow\mathbb{P}^{m|n} $ which $ \tilde{H}:\mathbb{P}^{m} \times \mathbb{R} \longrightarrow\mathbb{P}^{m} $ is a projection map and $ H^*:\mathcal{O}_{\mathbb{P}^{m|n}} \longrightarrow \mathcal{O}_{\mathbb{P}^{m|n}\times \mathbb{R}^{1|0}}$ is a sheaf homomorphism induced by the local homorphisms $ \mathcal{O}_{{U_{\alpha}^{m|n}}} \longrightarrow \mathcal{O}_{{U_{\alpha}^{m|n}}\times \mathbb{R}^{1|0}} $ with:
\begin{equation*}
x_{k} \longmapsto x_{k}, \qquad e_{l} \longmapsto (1-t)e_{l}.
\end{equation*}
Then one may easily show that for this morphism, all above diagrams commute.
\end{proof}
This proposition shows that in the construction of projective superspaces, the odd variables do not play principal roles. Solving this problem
is our motivation for defining $\nu$-Projective spaces or generally $\nu$-Grassmannians. Before that, it is necessary to recall some basic concepts.
\\
A  $\nu$-domain with dimension $ p|q $ is a super ringed space 
\begin{equation*}
	\mathbb{R}^{p|q}:=(\mathbb{R}^{p},\mathcal{O}),\qquad\mathcal{O}=\mathbf{C}^{\infty}_{\mathbb{R}^{p}} \otimes_{\mathbb{R}} \wedge \mathbb{R}^{q}, \qquad p, q \in \mathbb{N},
\end{equation*}
with an odd involution  $\nu$, i.e.
\begin{equation*}
	\nu: \mathcal{O} \longrightarrow \mathcal{O}, \qquad \nu(\mathcal{O}^{ev}) \subseteq \mathcal{O}^{odd}, \qquad \nu(\mathcal{O}^{odd}) \subseteq \mathcal{O}^{ev}, \qquad \nu^{2}=id.
\end{equation*}
In addition, $\nu$ is a homomorphism between $\mathbb{C}^{\infty}$-modules.
\\
Let $k$, $l$, $m$ and $n$ be non-negative integers with $k<m$ and $l<n$. For convenience from now on, we write $ Gr(k,m) $ as $ Gr_{k}^{m} $ , and set $p=k(m-k)+l(n-l)$ and $q=k(n-l)+l(m-k)$.
\\
A \textit{real} $\nu$-\textit{Grassmannian}, $_{\nu}Gr_{k|l}(m|n)$, or shortly $_{\nu}Gr= (Gr_{k}^{m} \times Gr_{l}^{n}, \mathcal{G})$, is a real superspace obtained by gluing $\nu$-domains $(\mathbb{R}^{p}, \mathcal{O})$ of dimension $p|q$.
\\
Here, we need to set some notations that are useful later.
\\
Let $I$ be a $k|l$ multi-index, i.e., an increasing finite sequence of $\{1, \cdots, m+n\}$ with $k+l$ elements. So one may set
\begin{equation*}
I:= \{i_{a}\}_{a=1}^{k+l}.
\end{equation*}
A standard $(k|l) \times (m|n)$ supermatrix, say $T$, may be decomposed into four blocks as follows:
$$\left[ \! \! \!
\begin{tabular}{ccc}
$A_{k\times m}$   &  $\vline$  & $B_{k \times n}$\\
---  ---  ---             &  $\vline$  & ---  ---  ---          \\
$C_{l \times m}$   &   $\vline$  & $D_{l \times n}$
\end{tabular}
\! \! \! \right].$$
The upper left and lower right blocks are filled by even elements. These two blocks together form the even part of $T$. The upper right and lower left blocks are filled by odd elements and form the odd part of $T$.
\\
The columns with indices in $I$ together form a minor denoted by $M_I(T)$. 
\\
A pseudo-unit matrix $id_{I}$ corresponding to $k|l$ multi-index $I$, is a $(k|l) \times (k|l)$ matrix whose all entries are zero except on its main diagonal that are $1$ or $1\nu$, where $1\nu$ is a formal expression used as odd unit. For each open subset $U$ of $\mathbb{R}^p$ and each $z \in \mathcal{O}(U)$, we also need the following rules:
\begin{equation*}
z.(1\nu):= \nu(z), \qquad (1\nu)(1\nu)=1.
\end{equation*}
So for each $I$, one has
\begin{equation*}
id_{I}.id_{I}= id.
\end{equation*}
As a result, for each $I$ and each $(k|l) \times(k|l)$ supermatrix $T$, we can see that
\begin{equation*}
T=(T.id_{I}).id_{I}.
\end{equation*}
The following steps may be taken in order to construct the structure sheaf of $_{\nu}Gr$:
\\
Step1: For each $k|l$ multi-index $I$, consider the $\nu$-domain $(V_{I}, \mathcal{G}_{I})$.
\\ 
Step2: Corresponding to each $I$, consider a $(k|l) \times (m|n)$ supermatrix $A^{I}$ which its columns with indices in $I$ together form $id_{I}$. The formal expression $1\nu$ appears when a diagonal entry of $id_{I}$ places in odd part of $A^{I}$.
\\
In addition, the other columns of $A^{I}$, from left to right, and each one from up to down, are filled by even and odd coordinates of $\nu$-domain $\mathcal{G}_{I}$, i.e., 
$x_1, ... , x_k, e_1, ... , e_l, ... , x_{(m- k-1)k+1}, ... , x_{(m- k)k},  e_{(m- k-1)l+1},$ ...$, e_{(m- k)l} ,e_{(m- k)l+1}, ... ,e_{(m- k)l+ k}, x_{(m- k)k+1}, ... , x_{(m- k)k+ l}, ... $ .
respectively. Afterwards, each entry, say $a$, that is in a block with opposite parity is replaced by $\nu(a)$.
\\
As an example, consider $_{\nu}Gr_{2|2}(3|3)$ with $I=\{ 1, 2, 3, 6\}$. Then one has
$$\left[ \! \! \! \!
\begin{tabular}{ccccccc}
1   & 0 & 0          & ; & $\nu x_{1}$ & $e_{3}$ & 0\\
0   & 1 & 0          & ; & $\nu x_{2}$ & $e_{4}$ & 0\\
--   &-- & --         & --& --  --            &      --  -- & --\\
0   & 0 & $1\nu$ & ; & $\nu e_{1}$  & $x_{3}$ & 0
\\
0   & 0 & 0          & ; & $\nu e_{2}$ & $x_{4}$ & 1
\end{tabular}
\! \! \! \! \right].$$
The columns of $A^{I}$ with indices in $I$ together form the following supermatrix:
$$M_{I}(A^{I}):= id_{I}=\left[ \! \! \! \!
\begin{tabular}{ccccccc}
1   & 0 & 0          & ; & 0\\
0   & 1 & 0          & ; & 0\\
--   &-- & --         & --& --\\
0   & 0 & $1\nu$ & ; & 0
\\
0   & 0 & 0          & ; & 1
\end{tabular}
\! \! \! \! \right].$$
For each pair multi-indices $I$ and $J$, define the set $V_{IJ}$ to be the largest subset of  $V_{I}$ on which $M_{J}(A^{I}).id_{J}$  is invertible on it.
\\
Step3: On $V_{IJ}$, the equality
\begin{equation*}
\big(M_{J}(A^{I}).id_{J}\big)^{-1}.A^{I} = A^{J},
\end{equation*}
defines a correspondence between even and odd coordinates of $V_{J}$ and rational expressions in $\mathcal{G}_{I}$ appear as corresponding entries of matrices on the two sides of the equality. By (\cite{Varadarajan}, Th 4.3.1), one has a unique homomorphism
\begin{equation*}
\varphi_{IJ}^{*}: \mathcal{G}_{J|_{V_{JI}}} \longrightarrow \mathcal{G}_{I|_{V_{IJ}}}.
\end{equation*}
Step4: The homomorphisms $\varphi_{IJ}^{*}$ satisfy the gluing conditions, i.e., for each $I$, $J$ and $K$, we have
\begin{itemize}
	\item[1)]
	$\varphi_{II}^{*}= id$.
	\item[2)]
	$\varphi_{IJ}^{*} \circ \varphi_{JI}^{*} = id$.
	\item[3)]
	$\varphi_{IK}^{*} \circ \varphi_{KJ}^{*} \circ \varphi_{JI}^{*}= id$.
\end{itemize}
In the first condition, $\varphi_{II}^{*}$ is defined by the following equality:
\begin{align*}
\big(M_{I}(A^{I}).id_{I}\big)^{-1}.A^{I} = \big(id_{I}.id_{I}\big)^{-1}.A^{I}.
\end{align*}
Since $\big(id_{I}.id_{I}\big)^{-1}= id$, we have $\varphi_{II}^{*}= id$.
\\
For the last condition, note that $\varphi_{KJ}^{*} \circ \varphi_{JI}^{*}$ is obtained from the equality
\begin{equation*}
\Big( M_{I}\Big( \big(M_{J}(A^{K}) .id_{J}\big)^{-1}.A^{K}\Big).id_{I}\Big)^{-1}\Big( \big(M_{J}(A^{K}) .id_{J}\big)^{-1}.A^{K}\Big)= A^{I}.
\end{equation*}
For the left hand side of this equality, one has
\begin{align*}
&\Big(\big(M_{J}(A^{K}) .id_{J}\big)^{-1}. M_{I}(A^{K}).id_{I}\Big)^{-1}\Big( \big(M_{J}(A^{K}) .id_{J}\big)^{-1}.A^{K}\Big)
\\
=&\big(M_{I}(A^{K}).id_{I}\big)^{-1} \big(M_{J}(A^{K}) .id_{J}\big)\big(M_{J}(A^{K}) .id_{J}\big)^{-1}. A^{k}
\\
=&\big(M_{I}(A^{K}).id_{I}\big)^{-1}.A^{K}= A^{I}.
\end{align*}
Thus the third condition is established.
\\
The second condition may results from other conditions as follows:
\begin{align*}
&\varphi_{IJ}^{*} \circ \varphi_{JI}^{*}=\varphi_{II}^{*},
\\
&\varphi_{II}^{*}=id.
\end{align*}
\section{Super vector bundles}
Here, we recall the definition of super vector bundles and their homomorphisms. Then, we introduce a canonical super vector bundle over $\nu$-Grassmannian.
\begin{definition}
By a \textit{super vector bundle} $\mathcal{E}$ of rank $k|l$ over a supermanifold $(M, \mathcal{O})$, we mean a sheaf of $\mathbb{Z}_{2}$-graded $\mathcal{O}$-modules on $M$ which locally is a free $k|l$ module.
\end{definition}
In other words, there exists an open cover $\{U_{\alpha}\}_{\alpha}$ of $M$ such that 
\begin{equation*}
\mathcal{E}(U_{\alpha}) \simeq \big(\mathcal{O}(U_{\alpha}) \big)^{k} \oplus \big(\pi\mathcal{O}(U_{\alpha}) \big)^{l},
\end{equation*}
or equivalently,
\begin{equation*}
\mathcal{E}(U_{\alpha}) \simeq \mathcal{O}(U_{\alpha}) \otimes_{\mathbb{R}} \big(\mathbb{R}^{k} \oplus \pi(\mathbb{R}^{l}) \big).
\end{equation*}
For example, let $\mathcal{O}_{M}^{k|l} :=(\oplus_{i=1}^{k} \mathcal{O}) \oplus (\oplus_{j=1}^{l} \pi \mathcal{O})$ where $\pi \mathcal{O}$ is an $\mathcal{O}$-module which satisfies
\begin{equation*}
(\pi \mathcal{O})^{ev} = \mathcal{O}^{odd}, \qquad (\pi \mathcal{O})^{odd} = \mathcal{O}^{ev}.
\end{equation*}
The right multiplication is the same as in $\mathcal{O}$ and the left multiplication is as follows:
\begin{equation*}
z(\pi w):= (-1)^{p(z)}\pi(zw),
\end{equation*}
where $\pi w$ is an element of $\pi\mathcal{O}$.
\\
Hence, $\mathcal{O}_{M}^{k|l}$ is a super vector bundle over the supermanifold $M$.
\\
Let $\mathcal{E}$ and $\mathcal{E}^{\prime}$ be two super vector bundles over a supermanifold $(M, \mathcal{O})$. By a homomorphism from $\mathcal{E}$ to $\mathcal{E}^{\prime}$, we mean an even sheaf homomorphism $\tau: \mathcal{E} \longrightarrow \mathcal{E}^{\prime}$.
\\
Each super vector bundle over $M$ isomorphic to $\mathcal{O}_{M}^{k|l}$, is called a \textit{trivial super vector bundle} of rank $k|l$.
\subsection{Canonical super vector bundle over $\nu$-Grassmannian}
Let $I$ be a $k|l$ multi-index and let $(V_{I}, \mathcal{G}_{I})$ be a $\nu$-domain. Consider the trivial super vector bundle
\begin{equation*}
\Gamma_{I}:= \mathcal{G}_{I} \otimes_{\mathbb{R}} \big(\mathbb{R}^{k} \oplus \pi(\mathbb{R}^{l}) \big) = \mathcal{G}_{I} \otimes_{\mathbb{R}} \mathbb{R}^{k|l}.
\end{equation*}
By gluing these super vector bundles through suitable homomorphisms, one may construct a super vector bundle $\Gamma$ over $\nu$-Grassmannian $_{\nu}Gr$. For this, consider a basis $\{e_{1}, \cdots, e_{k}, f_{1}, \cdots, f_{l}\}$ for $\mathbb{R}^{k|l}$ and set
\begin{equation*}
m:= \big(M_{J}(A^{I}) .id_{J}\big)^{-1},
\end{equation*}
where $\big(M_{J}(A^{I}) .id_{J}\big)^{-1}$ is introduced in subsection \ref{grassmannian}. Gluing morphisms are defined as follows:
\begin{equation*}
\psi_{IJ}^{*}: \Gamma_{J|_{V_{JI}}} \longrightarrow \Gamma_{I|_{V_{IJ}}},
\end{equation*}
\begin{equation*}
a \otimes e_{i} \longmapsto \varphi_{IJ|_{V_{JI}}}^{*}(a) \Big( \sum_{t \leq k} m_{it} \otimes e_{t} + \sum_{t > k}m_{it}\otimes f_{t}\Big),
\end{equation*}
where the elements $m_{it}$ are the entries of the $i$-th column of the supermatrix $m$. The morphisms $\psi_{IJ}^{*}$ satisfy the gluing conditions. So $\Gamma_{I}^,$s may glued together to form a super vector bundle denoted by $\Gamma$.
\subsection{Gauss morphisms}
In common geometry, a Gauss map is defined as a map from the total space of a vector bundle, say $\xi$, to a Euclidean space such that its restriction to any fiber is a monomorphism. Equivalently, one may consider a $1-1$ strong bundle map from $\xi$ to a trivial vecor bundle. The Gauss map induces a homomorphism between the vector bundle and the canonical vector bundle on a Grassmannian $Gr_{k}^n$ with sufficiently large value of $n$. A simple method for constructing such a map is the use of coordinate representation for $\xi$. In this section, for constructing a Gauss morphism of a super vector bundle, one may use the same method.
\begin{definition}
A super vector bundle  $\mathcal{E}$ over a supermanifold  $(M, \mathcal{O})$  is of finite type if there is a finite open cover  $\{U_{\alpha}\}_{\alpha=1}^{t}$  for $ M $ such that for each $ \alpha $, the restriction of  $\mathcal{E}$  to $U_{\alpha}$   is trivial, i.e., there exist isomorphisms as
\begin{equation*}
\psi_{\alpha}^*: \mathcal{E}|_{U_{\alpha}} \overset{\simeq}{\longrightarrow} \mathcal{O}|_{U_{\alpha}} \otimes _{\mathbb{R}} \mathbb{R}^{k|l}.
\end{equation*}
\end{definition} 
\begin{definition}
A Gauss morphism of $\mathcal{E}$ is a homomorphism from $\mathcal{E}$ to the trivial super vector bundle over $(M, \mathcal{O})$ so that its kernel is trivial.
\end{definition}
Let $\{e_{1}, \cdots, e_{k}, f_{1}, \cdots, f_{l}\}$ be a basis for $\mathbb{R}^{k|l}$ so that $\{e_{i}\}$ and $\{f_{j}\}$ are respectively bases for $\mathbb{R}^{k}$ and $\pi(\mathbb{R}^{l})$, then $B:=\{1 \otimes e_{1}, \cdots, 1 \otimes e_{k}, 1 \otimes f_{1}, \cdots, 1 \otimes f_{l}\}$ is a generator for the $\mathcal{O}(U_{\alpha})$-module, $\mathcal{O}(U_{\alpha}) \otimes_{\mathbb{R}} \mathbb{R}^{k|l}$.
\\
Set
\begin{equation}\label{salpha}
s_{i}^{\alpha}:= \psi_{\alpha}^{*^{-1}} (1 \otimes e_{i}), \qquad t_{j}^{\alpha}:= \psi_{\alpha}^{*^{-1}} (1 \otimes f_{j}).
\end{equation}
So $(\psi_{\alpha}^{*})^{-1} (B)$ is a generator for $\mathcal{E}(U_{\alpha})$ as an $\mathcal{O}(U_{\alpha})$-module.
\\
Choose a partition of unity $\{\rho_{\alpha}\}_{\alpha=1}^{t}$ subordinate to the covering $\{U_{\alpha}\}_{\alpha=1}^{t}$. Considering $s$ as a global section of $\mathcal{E}(M)$, we can write 
\begin{equation*}
s=\sum_{\alpha=1}^{t} \rho_{\alpha}r_{\alpha}(s),
\end{equation*}
where $r_{\alpha}$ is the restriction morphism. In addition, one has
\begin{equation*}
r_{\alpha}(s)=\sum_{i=1}^{k} \lambda_{i}^{\alpha}s_{i}^{\alpha} + \sum_{j=1}^{l} \delta_{j}^{\alpha}t_{j}^{\alpha}, \qquad \lambda_{i}^{\alpha}, \delta_{j}^{\alpha} \in \mathcal{O}(U_{\alpha}).
\end{equation*}
By the last two equalities, we have
\begin{equation*}
s=\sum_{\alpha=1}^{t} \rho_{\alpha} \Big( \sum_{i=1}^{k} \lambda_{i}^{\alpha}s_{i}^{\alpha} + \sum_{j=1}^{l} \delta_{j}^{\alpha}t_{j}^{\alpha}\Big) = \sum_{\alpha=1}^{t}\sum_{i=1}^{k} \sqrt{\rho_{\alpha}}\lambda_{i}^{\alpha}.
\sqrt{\rho_{\alpha}}s_{i}^{\alpha}+ \sum_{\alpha=1}^{t}\sum_{j=1}^{l} \sqrt{\rho_{\alpha}}\delta_{j}^{\alpha}.\sqrt{\rho_{\alpha}}t_{j}^{\alpha},
\end{equation*}
where $\sqrt{\rho_{\alpha}}s_{i}^{\alpha}$ and $\sqrt{\rho_{\alpha}}t_{j}^{\alpha}$ are even and odd sections of $\mathcal{E}(M)$ respectively, and $\sqrt{\rho_{\alpha}}\lambda_{i}^{\alpha}$ and $\sqrt{\rho_{\alpha}}\delta_{j}^{\alpha}$
are sections of $\mathcal{O}(M)$. So $A:= \{\sqrt{\rho_{\alpha}}s_{i}^{\alpha}\}_{\alpha, i} \cup \{\sqrt{\rho_{\alpha}}t_{j}^{\alpha}\}_{\alpha, j}$ is a generating set of $\mathcal{E}(M)$.
\\
Now, for each $\alpha$, consider the following monomorphism between $\mathcal{O}(U_{\alpha})$-modules:
\begin{align*}
i_{\alpha}: \mathcal{O}(U_{\alpha}) \otimes_{\mathbb{R}} \mathbb{R}^{k|l} &\longrightarrow \mathcal{O}(U_{\alpha}) \otimes_{\mathbb{R}} \mathbb{R}^{tk|tl},\\
1 \otimes e_{i} \longmapsto & 1 \otimes e_{(\alpha-1)k+i},\\
1 \otimes f_{j} \longmapsto & 1 \otimes f_{(\alpha-1)l+j}.
\end{align*}
Set
\begin{equation}
g(s):= \sum_{\alpha=1}^{t} \rho_{\alpha}. i_{\alpha} \circ \psi_{\alpha}^* \circ r_{\alpha}(s).
\end{equation}
It is easy to see that $g$ is a Gauss morphism of $\mathcal{E}(M)$.
\subsection{Gauss supermatrix}\label{Gauss sup}
Now, we are going to obtain the matrix of the Gauss supermap $g$.
\\
\begin{definition}
By a Gauss supermatrix associated to the super vector bundle $\mathcal{E}$, we mean a supermatrix, say $G$, which is obtained as follows with respect to the generating set $A$:
\begin{equation*}
g\big(\sqrt{\rho_{\beta}}s_{j}^{\beta}\big) = \sum_{\alpha=1}^{t} \rho_{\alpha}. i_{\alpha} \circ \psi_{\alpha}^* \circ r_{\alpha}\big( \sqrt{\rho_{\beta}}s_{j}^{\beta}\big),
\end{equation*}
where $g$ is a Gauss morphism of $\mathcal{E}$.
\end{definition}
By \eqref{salpha}, we have
\begin{align}\label{grho}
g\big(\sqrt{\rho_{\beta}}s_{j}^{\beta}\big) &= \sum_{\alpha=1}^{t} \rho_{\alpha}. i_{\alpha} \circ \psi_{\alpha}^* \circ \psi_{\beta}^{*^{-1}}\big( \sqrt{\rho_{\beta}}e_{j}\big) \nonumber
\\
&=\sum_{i=1}^{k}\sum_{\alpha=1}^{t} \rho_{\alpha} \sqrt{\rho_{\beta}} a_{ij}^{\alpha\beta}e_{(\alpha-1)k+i}+\sum_{i=1}^{l}\sum_{\alpha=1}^{t}\rho_{\alpha} \sqrt{\rho_{\beta}}a_{(k+i)j}^{\alpha\beta}f_{(\alpha-1)l+i},
\end{align}
where $[a_{ij}^{\alpha\beta}]$ is a matrix of $\psi_{\alpha}^* \circ \psi_{\beta}^{*^{-1}}$ relative to the generator $B$. The natural ordering on $\{i_{\alpha}e_{i}\}_{\alpha, i}$ and $\{i_{\alpha}f_{s}\}_{\alpha, s}$ induces an ordering on their coefficients in \eqref{grho}. Let G be a $tk|tl \times tk|tl$ standard supermatrix. Fill the even and odd top blocks of $G$ by these coefficients according to their parity from left to right along the $\big((\beta-1)k+j\big)$-th row, $1 \leq j \leq k$, $1 \leq \beta \leq t$. Similarly, by coefficients in the decomposition of $g\big(\sqrt{\rho_{\beta}}t_{r}^{\beta}\big)$, one may fill the odd and even down blocks of $G$ along the $\big((\beta -1)k+r\big)$-th row, $1 \leq r \leq l$, $1 \leq \beta \leq t$.
\begin{example}
	For $k|l=2|1$ and a suitable covering with two elements, i.e., $t=2$, we have
	\begin{equation*}
	g\big(\sqrt{\rho_{2}}s_{2}^{2}\big)= \rho_{1} \sqrt{\rho_{2}} a_{12}^{12}e_{1} +\rho_{1} \sqrt{\rho_{2}} a_{22}^{12}e_{2}+\rho_{2} \sqrt{\rho_{2}} a_{12}^{22}e_{3}+\rho_{2} \sqrt{\rho_{2}} a_{22}^{22}e_{4}+\rho_{1} \sqrt{\rho_{2}} a_{32}^{12}f_{1}+\rho_{2} \sqrt{\rho_{2}} a_{32}^{22}f_{2}.  
	\end{equation*}
	Then the $4$-th row of the associated supermatrix $G$ is as below:
	$$\left[ \! \! \! \!
	\begin{tabular}{ccccccc}
	&                                                          &                                                           &                                                                                               \vdots 
	& ;                          &                                                            &
	\\
	$\rho_{1} \sqrt{\rho_{2}} a_{12}^{12}$   & $\rho_{1} \sqrt{\rho_{2}} a_{22}^{12}$ & $\rho_{2} \sqrt{\rho_{2}} a_{12}^{22}$ & $\rho_{2} \sqrt{\rho_{2}} a_{22}^{22}$  & ;                         &$\rho_{1} \sqrt{\rho_{2}} a_{32}^{12}$ & $\rho_{2} \sqrt{\rho_{2}} a_{32}^{22}$\\
	-- --         -- --            -- --                        &     -- --       -- --    -- --                         & -- --     -- --     -- --                              & -- --   -- --  -- --                                                         & -- --   -- --   -- --    &-- --  -- -- -- --                                        &   -- --   -- --   -- --  \\
	&                                                          &                                                           &                                                                                               \vdots 
	& ;  &                                                             &
	\\
	\end{tabular}
	\! \! \! \! \right].$$
\end{example}
On the other hand, one may consider a covering $\{U_{\alpha}\}_{\alpha}$ so that for each $\alpha$, we have an isomorphism
\begin{equation}\label{isomorphism}
\mathcal{O}(U_{\alpha}) \overset{\simeq}{\longrightarrow} \mathbf{C}^{\infty}(\mathbb{R}^{m}) \otimes _{\mathbb{R}} \wedge \mathbb{R}^{n}.
\end{equation}
Let $\nu$ be an odd involution on $\mathbf{C}^{\infty}\big(\mathbb{R}^{(k^{2}+l^{2})(t-1)}\big) \otimes _{\mathbb{R}} \wedge \mathbb{R}^{2kl(t-1)}$ preserving $\mathbf{C}^{\infty}(\mathbb{R}^{m}) \otimes _{\mathbb{R}} \wedge \mathbb{R}^{n}$ as a subalgebra. Thus, it induces an odd involution on $\mathcal{O}(U_{\alpha})$ through the isomorphism \eqref{isomorphism} which is denoted by the same notation $\nu$.
\begin{theorem}\label{Sigma}
	Let $\cal{E}$ be a super vector bundle over a supermanifold $(M, \cal{O})$ and let $G$ be a Gauss supermatrix associated to $\cal{E}$. Then the Gauss supermatrix induces a homomorphism from $\mathcal{G}$, the structure sheaf of   $_{\nu}Gr$, to $\mathcal{O}$.
\end{theorem}
\begin{proof}
Let $h$ be an element of $\mathcal{G}(Gr_k^m\times Gr_l^n)$, where $ m=tk$, $ n=tl$, and $\{\rho_{I}^{\prime}\}$ be a partition of unity subordinate to the covering $\{V_{I}\}_{I \subseteq \{1, \cdots, t(k+l)\}}$, then one has
\begin{equation}
h= \sum_{\substack{I \subseteq \{1, \cdots, t(k+l)\},\\ |I|=k+l}} \rho_{I}^{\prime}.h|_{V_{I}}.
\end{equation}
Consider the rows of $G$ with indices in $I$ as a $(k+l) \times t(k+l)$ supermatrix and name it $G(I)$. Then multiply it by $id_{I}$ from left, i.e., $id_{I}. G(I)$ and delete the columns with indices in $I$, we get
$$\left[ \! \! \! \!
\begin{tabular}{ccccccc}
$y_{11}^{I}$  $\qquad$     & $\cdots$    $\qquad$       &$y_{1\big((t-1)(k+l)\big)}^{I}$\\
$\vdots$          $\qquad$    &  $\ddots$   $\qquad$        &$\vdots$ \\
$y_{(k+l)1}^{I}$ $\qquad$ &  $\cdots$   $\qquad$       &$y_{(k+l)\big((t-1)(k+l)\big)}^{I}$
\end{tabular}
\! \! \! \! \right].$$
Note that all entries of this supermatrix are sections of $\mathcal{O}(M)$.
\\
Let $A^{I}$  be the matrix introduced in subsection \ref{grassmannian} and let $x_{ij}^{I}$ be its entry out of $M_{I}(A^{I})=id_{I}$. Then the correspondence $x_{ij}^{I} \longmapsto y_{ij}^{I}$ defines a homomorphism 
\begin{equation*}
\varphi_{I}^{*}: \mathcal{G}_{I}(V_{I}) \longrightarrow \mathcal{O}(M).
\end{equation*}
Now, for each global section $h$ of $\mathcal{G}$, one may define 
\begin{equation*}
\widetilde{h} = \sum_{\substack{I \subseteq \{1, \cdots, t(k+l)\},\\ |I|=k+l}} \varphi_{I}^{*}( \rho_{I}^{\prime}.h|_{V_{I}}).
\end{equation*}
Then the correspondence 
\begin{equation}\label{sigma}
\sigma^*: h \longmapsto \widetilde{h}
\end{equation}
is a  well-defined homomorphism from $\mathcal{G}(Gr_{k}^{tk} \times Gr_{l}^{tl})$ to $\mathcal{O}(M)$ and so induces a smooth map $\widetilde{\sigma}$, from $M$ to $Gr_{k}^{tk} \times Gr_{l}^{tl}$ \cite{Roshandel}.
\end{proof}
The homomorphism $\sigma=(\widetilde{\sigma},\sigma^*)$ is called the associated morphism with the Gauss morphism $g$.
\subsection{Pullback of the canonical super vector bundle}
\begin{definition}
Let  $\sigma=(\widetilde{\sigma},\sigma^*)$  is a morphism from a supermanifold  $(M,\mathcal{O}) $  to  $ Gr_{k|l}^{m|n} $. One can define a  $\mathcal{G}$-module structure on  $\mathcal{O}(M)$ as follows:
\begin{equation*}
a * b:= \sigma^*(a).b, \qquad a \in \mathcal{G}(Gr_{k}^{m} \times Gr_{l}^{n}), \qquad b \in \mathcal{O}(M).
\end{equation*}
The sheaf $\mathcal{O} \otimes _{\mathcal{G}}^{\sigma} \Gamma$  is called the structure sheaf of the pullback of  $ \gamma_{k|l}^{m|n}$ along  $\sigma$ .
\end{definition}
\begin{theorem}\label{sigms iso}
	Let $\sigma$ be the associated morphism introduced above\eqref{sigma}. Then, the super vector bundle $\mathcal{E}$ and the pullback of $\Gamma$ (the canonical super vector bundle over $_{\nu}Gr$) along $\sigma$ are isomorphic.
\end{theorem}
\begin{proof}
We show that the sheaf $\mathcal{O} \otimes _{\mathcal{G}}^{\sigma} \Gamma$ is isomorphic to $\mathcal{E}$. Let $s^{\prime}$ be a global section on $\Gamma$. One has
\begin{equation*}
s^{\prime}= \sum_{\substack{I \subseteq \{1, \cdots, t(k+l)\},\\ |I|=k+l}} \rho_{I}^{\prime}. r_{I}^{\prime}(s^{\prime}),
\end{equation*}
where $\{\rho_{I}^{\prime}\}$ is the partition of unity of $_{\nu}Gr$ subordinate to the open cover $\{V_{I}\}$, and $r_{I}^{\prime}$ is the restriction morphism giving sections over $V_{I}$. On the other hand, one may write each section $r_{I}^{\prime}(s^{\prime})$ as below:
\begin{equation*}
r_{I}^{\prime}(s^{\prime}) = \sum_{j=1}^{k+l} h_{j}^{I} s_{j}^{\prime^{I}},
\end{equation*}
where $s_{j}^{\prime^{I}}$ are generators of $\Gamma(V_{I})$ and the coefficients $h_{j}^{I}$ are the sections of $\mathcal{G}_{I}$. Therefore, we can write 
\begin{equation*}
s^{\prime}= \sum_{\substack{I \subseteq \{1, \cdots, t(k+l)\},\\ |I|=k+l}} \sum_{j=1}^{k+l} (\rho_{I}^{\prime}h_{j}^{I})s_{j}^{\prime^{I}}.
\end{equation*}
Note that each row of $G$ is in correspondence with a section in the generator set $A$. So there is a morphism from the pullback of $\Gamma$ to $\mathcal{E}$ as
\begin{align}\label{sigmatensor}
&\mathcal{O} \otimes_{\mathcal{G}}^{\sigma} \Gamma \longrightarrow \mathcal{E}, \nonumber
\\
&u \otimes s^{\prime} \longmapsto u.\delta(s^{\prime}),
\end{align}
where $\delta(s^{\prime})$ is 
\begin{equation*}
\sum_{\substack{I \subseteq \{1, \cdots, t(k+l)\},\\ |I|=k+l}} \sum_{\substack{j=1}}^{k+l} \sigma^*(\rho_{I}^{\prime}h_{j}^{I}). s_{j}^{I}
\end{equation*}
and $s_{j}^{I}$ is the section corresponding to the $j$-th row of $G(I)$ (cf. subsection \ref{Gauss sup}).
\\
One may show that the morphism in \eqref{sigmatensor} is an isomorphism. To this end, first note that every locally isomorphism between two sheaves of $\mathcal{O}$-modules with the same rank is a globally isomorphism. Also for the super vector bundle $\Gamma$ of rank $k|l$ over $\mathcal{G}$, one can write a locally isomorphism
\begin{equation*}
\mathcal{O} \otimes_{\mathcal{G}}^{\sigma} \Gamma \overset{\simeq}{\longrightarrow} \mathcal{O} \otimes_{\mathbb{R}}\mathbb{R}^{k|l},
\end{equation*}
because for each sufficiently small open set $V$ in $Gr_{k}^{tk} \times Gr_{l}^{tl}$ one can write
\begin{equation*}
\Gamma(V) \simeq \mathcal{G}(V) \otimes_{\mathbb{R}}\mathbb{R}^{k|l} 
\end{equation*}
and then
\begin{equation*}
\mathcal{O} \big(\widetilde{\sigma}^{-1}(V)\big) \otimes_{\mathcal{G}}^{\sigma} \Gamma(V) \simeq \mathcal{O}\big(\widetilde{\sigma}^{-1}(V)\big) \otimes_{\mathcal{G}}^{\sigma} \mathcal{G}(V) \otimes_{\mathbb{R}} \mathbb{R}^{k|l}.
\end{equation*}
This shows that the morphism in \eqref{sigmatensor} may be represented locally by the following isomorphism:
\begin{equation}
\mathcal{O} \otimes \mathcal{G}\otimes \mathbb{R}^{k|l} \to \mathcal{O}\otimes \mathbb{R}^{k|l}.
\end{equation}
Thus \eqref{sigmatensor} defines a global isomorphism.
\end{proof} 
\subsection{Homotopy properties of Gauss supermaps and their associated morphisms}

Let $\mathcal{O}_{M}^{m|n}$ and $\mathcal{O}_{M}^{m^{\prime}|n^{\prime}}$ be two trivial super vector bundles where $m^{\prime}=2m-k$, $n^{\prime}=2n-l$; then, one can write the inclusion homomorphisms
\begin{equation*}
J^{e}, J^{o}, J: \mathcal{O}(M) \otimes_{\mathbb{R}} \mathbb{R}^{m|n} \longrightarrow \mathcal{O}(M) \otimes_{\mathbb{R}} \mathbb{R}^{2m|2n}
\end{equation*}
by the conditions
\begin{align*}
J^{e}: &1 \otimes e_{i} \longmapsto 1 \otimes e_{2i},               &J^{o}: 1 \otimes e_{i}  \longmapsto 1 \otimes e_{2i-1}, &\qquad \quad J: 1 \otimes e_{i} \longmapsto 1 \otimes e_{i}, \\
&1 \otimes f_{j} \longmapsto 1 \otimes f_{2j},                & 1 \otimes f_{j} \longmapsto 1 \otimes f_{2j-1},             &\qquad \qquad \quad 1 \otimes f_{j} \longmapsto 1 \otimes f_{j}.
\end{align*}
Now, let $(V_I, \mathcal{G}_I)$ be $\nu$-domains introduced in subsection \ref{grassmannian}. In addition assume $(W_J, \mathcal{G}_J)$ be $\nu$-domains of dimension $2p|2q$. For each $k|l$ multi-index $I=\{i_{1}, \cdots, i_{k+l}\} \subset \{1, ..., m+n\}$, one can associate the following multi-indices
\begin{equation*}
I^{e}:=\{2i_{1}, \cdots, 2i_{k+l}\},
\end{equation*}
\begin{equation*}
I^{o}:=\{2i_{1}-1, \cdots, 2i_{k+l}-1\},
\end{equation*}
\begin{equation*}
\bar{I}:=\{i_{1}, \cdots, i_{b}, i_{b+1}+m-k, \cdots, i_{k+l}+m-k\},
\end{equation*}
where $i_a \in I$ ,$1 \leq i_a \leq k+l$, and $i_b$ is an element of $I$ for which $i_b \leq m \leq i_{b+ 1}$.
\\
So the maps $J^{e}$, $J^{o}$ and $J$ induce the homomorphisms
\begin{equation*}
\bar{J}^{e}, \bar{J}^{o}, \bar{J}: _{\nu}Gr(m|n) \longrightarrow _{\nu}Gr(m^{\prime}|n^{\prime}).
\end{equation*}
In fact, $(\bar{J}^{e})^*|_{W_{I^{e}}}$ is obtained by
\begin{align*}
&\mathcal{G}_{I^{e}}(W_{I^{e}}) \longrightarrow \mathcal{G}_{I}(V_{I}),
\\
&\left\{
\begin{array}{rl}
y_{i(2j-1)} \longmapsto x_{ij}, \quad i=1, \cdots, k+l, \quad j=1, \cdots, m+n-k-l,
\\
\text{other generators} \longmapsto 0. \hspace{6cm}
\end{array}\right.
\end{align*}
\begin{theorem}
	Let $f, f_{1}:(M, \mathcal{O}) \longrightarrow _{\nu}Gr(m|n)$ are induced by the Gauss supermaps $g$ and $g_{1}$. Then, $\bar{J}f$ and $\bar{J}f_{1}$ induced by $Jg$ and $Jg_{1}$ are homotopic.
\end{theorem}
\begin{proof}
Consider the homomorphisms $J^{e}g$ and $J^{o}g_{1}$, with the induced homomorphisms $\bar{J}^{e}f$ and $\bar{J}^{o}f_1$. One can define a family of homomorphisms
\begin{align*}
&F_{t}: \mathcal{E}(M) \longrightarrow \mathcal{O} \otimes_{\mathbb{R}} \mathbb{R}^{2m|2n},
\\
&\varphi \longmapsto (1-t).(J^{e}g)(\varphi)+t.(J^{o}g_{1})(\varphi),
\end{align*}
where $F_{0}=J^{e}g$ and $F_{1}=J^{o}g_{1}$. By subsection \ref{Gauss sup}, a family of morphisms $\bar{F}_{t}$ from $(M,\mathcal{O})$ to $_{\nu}Gr(m^{\prime}|n^{\prime})$ are induced. Obviousely  $\bar{F}_{0}=\bar{J}^{e}f$ and $\bar{F}_{1}=\bar{J}^{o}f_{1}$; thus, $\bar{F}_{t}$ is a homotopy from $\bar{J}f$ to $\bar{J}f_1$.
\end{proof}
 \section{Universality}
At first we introduce the infinite Grassmannian $ Gr_k^{\infty} $ which is the  direct limit of finite dimensional Grassmannians $ Gr_k^m $.
\subsection{Grassmann Manifold $Gr_k^{\infty} $}\label{Gr-infty}
The Grassmann manifold,  $ Gr_k^n $, is a compact space containing all the k-dimensional vector subspaces in $ \mathbb{R}^n $. Each of these vector subspaces can be represented by a linearly independent $ k $-tuple basis $ {v_1,...,v_k} $. Any $ v_i $ is in the form of an ordered $ n $-tuple. With the help of inclusion mapping $\iota:\mathbb{R}^{n} \hookrightarrow \mathbb{R}^{n+1}$, these vectors can be embedded as $ (n+1) $-tuples $ \iota(v_i)= (v_{i1},...,v_{in},0) $ in $ \mathbb{R}^{n+1} $. Thus, a mapping like
\begin{equation*}
\iota:Gr_{k}^{n} \longrightarrow Gr_{k}^{n+1}
\end{equation*}
is induced, which is represented by  $ \iota $  for the sake of simplicity. Under this mapping  $ Gr_k^{n} $  can be considered as a subspace of  $ Gr_k^{n+1}  $. By standard coordinate map of a Grassmannian we mean coordinate maps which are introduced in [\cite{Manin}, page 9]. In this case, if $ (V'_I,x_1,...,x_{p+k})  $ for $ p=k(n-k) $ is a standard coordinate map of $ Gr_{k}^{n+1} $ corresponding to multi-indices $ I=\{i_1,...,i_k\}\subseteq\{1,...,n\} $, then  $ (V_I,x_1,...,x_p) $  for  $  \iota^{-1}(V'_I)=V_I $ is a standard coordinate map of $ Gr_k^{n} $.\\
From the sheaf-theoretic view, there is a sheaf $ (U_{I},\mathcal{O}_{I}):=(\mathbb{R}^p,\mathbf{C}^{\infty}_{\mathbb{R}^{p}}) $ for each coordinate map $ (V_I,x_1,...,x_p) $ that can be glued together through homomorphisms
\begin{equation}\label{phi}
\phi_{IJ}:(U_{IJ},\mathcal{O}_{J}|)\longrightarrow (U_{IJ},\mathcal{O}_{I}|)
\end{equation}
to construct $ Gr_k^{n} $.\\
In both view of points, by means of some $ \iota $, one may consider $ Gr^n_k $ as a subset of $ Gr^{n+1}_k $.\\
The union of all $ Gr_k^{n}$ for a constant $ k $ constitutes the infinite dimensional Grassmann manifold to which an inductive topology is described by means of a direct limit: 
\begin{equation*}
Gr_k^{\infty} =\bigcup_{k\leq n} Gr_k^n. 
\end{equation*}
Thus, each subset of  $ Gr_k^{\infty} $  is said to be open if its intersection with any $ Gr_k^{n} $ is an open set. Equivalently, in order to determine the topology of this set, a sequence of open sets like $ U_i\subseteq Gr_k^{k+i} $ where $(U_i)_{i\in \mathbb{N} \cup \{0\}}$   can be used such that for each i, we have
\begin{equation*}
\iota(U_i)\subseteq U_{i+1}.
\end{equation*}
In this case, on the union of a separated family of open sets $  U_i $s, consider the equivalence  $ \sim $  with the following criterion 
\begin{equation*}
a\sim b \leftrightarrow \iota({a})=b,
\end{equation*}
where  $ a\in U_i $  and  $ b\in U_{i+1} $. We define $ 	U^{\infty}_{\left( U_i\right) }:=\frac{\bigsqcup U_i}{\sim} $ as an open set in $ Gr_k^{\infty} $ and represent it abbreviately by $ U^\infty $ . These sets form a basis for the topology of  $ Gr_k^{\infty} $. Thus the following embedding can be defined for each i: 
\begin{align*}
p_i:&Gr_k^{k+i}\rightarrow Gr_k^{\infty},
\\&\quad a_i\mapsto[a_i].
\end{align*}
For the sake of simplicity, we represent the structure sheaf of $ Gr_k^{k+i} $  by $ \mathcal{O}_i $. 
\\
Now, using the inverse limit, a smooth sheaf structure can be defined over $ Gr_k^{\infty} $, which is denoted  by $ \underset{\leftarrow}{\mathcal{O}} $ . In fact, section $ f $ on the open set $ U^{\infty} $ is defined by the following rule:
\begin{equation*}
f\in \underset{\leftarrow}{\mathcal{O}}{(U^{\infty})}
\Longleftrightarrow f=(f_i)_i,\quad f_i\in \mathcal{O}_i(U_i),\quad \iota_{i}^*(f_{i+1})=f_i,
\end{equation*}
where, homorphisms $ \iota_{i}^*:\mathcal{O}_{i+1}\rightarrow \mathcal{O}_i $  are induced by mappings $ \iota_i:Gr_k^{k+i}\rightarrow Gr_k^{k+i+1} $ between the sheaves of the two manifolds. 
\\
Under this mapping, the structure sheaf $ \mathcal{O}_i $  enjoys a $ \mathcal{O}_{i+1} $-module structure.  
\begin{theorem}
	$ \underset{\leftarrow}{\mathcal{O}} $ is a sheaf.
\end{theorem} 
\begin{proof}
	For each two open sets $ U^{\infty} $ and $ V^{\infty} $, we have: 
	\begin{equation*}
	U^{\infty}\subseteq V^{\infty} \Longleftrightarrow U_i\subseteq V_i \qquad \forall i\in \mathbb{N}\cup \{0\}.
	\end{equation*}
	In this case, the restriction homomorphism can be defined as follows: 
	\begin{equation*}
	r_{( U_i)( V_i)}(f):=(r_{U_iV_i}(f_i))_i.
	\end{equation*}
	With this definition, it is easy to see that for all three open sets $ U^{\infty} $, $ V^{\infty} $, and $ W^{\infty} $, if   $ U^{\infty}\subseteq V^{\infty}\subseteq W^{\infty} $, then the following relationships are established: 
	\begin{align*}
	r_{U^\infty V^\infty}\circ r_{V^\infty W^\infty}&=r_{U^\infty W^\infty},\\
	r_{U^\infty U^\infty}&=id_{\underset{\leftarrow}{\mathcal{O}}{}_{U^\infty}}.
    \end{align*}
	Therefore, $ \underset{\leftarrow}{\mathcal{O}} $  is a pre-sheaf. To prove that it is a sheaf, it suffices to examine the local and global properties of sheaves. Consider two sections $ f $ and $ g $ on  $ Gr_k^{\infty} $ so that for any arbitrary point in the base space, they are equal in a neighborhood of that point. We know these sections are in the form of a sequence of sections $f=(f_i)  $  and  $ g=(g_i) $ , so for each $ i $ and any arbitrary point in  $ Gr_k^{k+i} $ , the sections $ f_i $ and $ g_i $ are equal in a neighborhood $ U_i $ of that point. Since $ Gr_k^{k+i} $  is a sheaf, according to the local property of sheaves, the sections $ f_i $ and $ g_i $ are equal on the whole $ Gr_k^{k+i} $. As a result, $ f $ and $ g $ are equal.\\
	Now consider an open cover  $\{U^{\infty}_{\alpha}\} $  for  $ Gr_k^{\infty} $  and sections  $f_\alpha\in \underset{\leftarrow}{\mathcal{O}}{} _{\alpha} $ so that for each  $ \alpha $  and  $ \beta $, we have
	\begin{equation*}
	r_{\alpha\beta,\beta}(f_\beta)=r_{\alpha\beta,\alpha}(f_\alpha),
	\end{equation*}
	where  $ r_{(U^{\infty}_{\alpha}\cap U^{\infty}_{\beta})U^{\infty}_{\beta}} $ are abbreviately represented by $ r_{\alpha\beta,\beta} $. Since each  $ f_\alpha $  is in the form of a sequence of sections like  $ f_{\alpha i} $, so for each $ i $, there is a section $ f_i $ and a restriction $ r_{\alpha i} $ in $ Gr_k^{k+i} $ to $ U_{\alpha i} $ such that  $ r_{\alpha i}(f_i)=f_{\alpha i} $. Therefore,  $ f=(f_i)_i $ is a section for which for each $ \alpha  $  we have 
	\begin{equation*}
	r_\alpha(f)=f_\alpha,
	\end{equation*}
	indicating that the global property for the pre-sheaf $  \underset{\leftarrow}{\mathcal{O}} $ is satisfied. 
\end{proof}	
Considering the structure defined, naturally a homomorphism from a finite-dimensional Grassmannian to an infinite-dimensional Grassmannian is induced by the embedding $ p_i $: 
\begin{equation*}
p_i^*:\underset{\leftarrow}{\mathcal{O}}\rightarrow \mathcal{O}_i
,\quad p_i^*((f_j))=f_i.
\end{equation*}
Under this mapping, the structure sheaf $ \mathcal{O}_i $  also has a $  \underset{\leftarrow}{\mathcal{O}} $-module structure.
\subsection{$\nu$-Grassmannian  $ _\nu Gr_{k|l}^{\infty} $}
To construct the $\nu$-Grassmannian $  _\nu Gr_{k|l}^{\infty}:=(Gr_k^\infty\times Gr_l^\infty , \underset{\leftarrow}{\mathcal{G}}) $, similar to the previous subsection, first its base space, i.e., $ Gr_k^{\infty}\times Gr_l^{\infty} $ is determined using the direct limit and then its sheaf structure is determined using the inverse limit. This requires that a natural homomorphism between the two $\nu$-Grassmannians   $ _\nu Gr_{k|l}^{m|n} $  and  $ _\nu Gr_{k|l}^{m'|n'} $  for  $ m\leq m' $  and  $ n\leq n' $  are examined.
\\ 
Firstly, since the structure of the base spaces is multiplicative, the mapping 
\begin{equation*}
\iota: Gr_k^m\times Gr_l^n\rightarrow  Gr_k^{m'}\times Gr_l^{n'}
\end{equation*}
is established between them.
\\
Secondly, note that the  $\nu$-Grassmannian  $ _\nu Gr_{k|l}^{m'|n'} $  is  constructed by gluing together standard  $\nu$-domains of dimension $ p'|q' $ where $ p'=p+k(m'-m)+l(n'-n) $ and $ q'=q+k(n'-n)+l(m'-m) $.
Similar to the normal case, if $ V'_I\cap \iota(Gr_k^m\times Gr_l^n) $ is nonempty for an arbitrary multi-index $ I $, then $ \iota^{-1}(V'_I) $ will be equal to $ V_I $.
In this case the correspondence 
\begin{align*}
&\mathcal{G'}_{I}(V'_{I}) \longrightarrow \mathcal{G}_{I}(V_{I}) \left\{
\begin{array}{rl}
&x_i\longmapsto x_i,\quad i=1,\cdots,p,\\
&e_i \longmapsto e_i,\quad i=1,\cdots,q,\qquad I\subseteq \{1,...,m\}\cup\{m'+1,...,m'+n\},\\
&\text{other generators} \longmapsto 0\hspace{0cm},
\end{array} \right.\\
&\mathcal{G}'_{I}(V'_{I}) \longrightarrow 0, \qquad\qquad\qquad\qquad\qquad\qquad\qquad\qquad I\nsubseteq \{1,...,m\}\cup\{m'+1,...,m'+n\}.
\end{align*}
induces a natural homomorphism as
\begin{equation}
\iota_I^*:\mathcal{G}'_I|({V'_{IJ})}\longrightarrow\mathcal{G}_I|({V_{IJ}})\label{a2}.
\end{equation}
Under this mapping, the structure sheaf   $ \mathcal{G}_I $  has a $ \mathcal{G}'_I $-module structure as well.
\\ 
It can be seen that the following diagram consisting of this homomorphism and the gluing maps between the  $ \nu $-domains for arbitrary and permissible multi-indices $ I $ and $ J $ is commutative;
\\
$$ \! \! \!
\begin{tabular}{ccc}
$\mathcal{G}'_J|(V'_{JI})$   &  $\stackrel{\iota_J^*}{\longrightarrow}  $  &   $\mathcal{G}_J|(V_{JI})$\\
${\varphi'}_{IJ}^* \downarrow $   &    &      $\varphi_{IJ}^* \downarrow $       \\
$\mathcal{G}'_I|(V'_{IJ})$   &   $ \stackrel{\iota_I^*}{\longrightarrow} $  &   $\mathcal{G}_I|(V_{IJ})$
\end{tabular}
\! \! \! $$
because the gluing maps  $ \varphi^*_{IJ} $  and $ {\varphi'}^*_{IJ} $  are exactly constructed using the invertible square supermatrices   $ M_J(A^I).id_J $ and  $ M'_J(A'^I).id_J $  of dimensions $ {k|l}\times {k|l} $  that are equal to each other and, on the other hand, the orders of the even generators $ x_1,...,x_p $ and the odd generators $ e_1,...,e_q $ do not change in the same columns in the two supermatrices. 
\\
Actually, in the representation supermatrix of  $\nu$-domain $\mathcal{G}'_J|_{V'_{JI}}$, there are columns like  $[^{[a']}_{[c']}]  $  and $ [^{[b']}_{[d']}] $  that do not present in the representation supermatrix  $\mathcal{G}_J|_{V_{JI}}$, and the other columns are same: 
$$A^J:=\left[ \! \! \!
\begin{tabular}{ccc}
$[A]_{k \times m}$   &  $\vline$  & $[B]_{k\times n}$\\
---  ---  ---             &  $\vline$  & ---  ---  ---          \\
$[C]_{l \times m}$   &   $\vline$  & $[D]_{l \times n}$
\end{tabular}
\! \! \! \right] ,\quad A'^J:= \left[ \! \! \!
\begin{tabular}{ccc}
$[A]_{k \times m}  [a']_{k \times {(m'-m)}}$  &  $\vline$  & $[B]_{k\times n}  [b']_{k \times {(n'-n)}}$ \\
---  ---  ---  ---  ---  ---  &  $\vline$  & ---  ---  ---  ---  ---  --- \\
$[C]_{l \times m}  [c']_{l \times (m'-m)}$   &   $\vline$  & $[D]_{l \times n}  [d']_{l \times (n'-n)}$
\end{tabular}
\! \! \! \right].$$
These columns have no effect on the gluing mappings in the above diagram because these mappings are derived based on the minors whose columns in the both supermatrices are the same. As a result, the above diagram commutes.
\\
The homomorphisms $ \iota_I^* $  induce the following natural homomorphisms by the induced equivalences $ {\varphi}_{IJ}^* $  and  $ {\varphi'}_{IJ}^* $, 
\begin{equation*}
\iota^*_{ii'|jj'}
:\mathcal{G}'(Gr_k^{m'}\times Gr_l^{n'})\rightarrow \mathcal{G}(Gr_k^m\times Gr_l^n) ,\qquad  m\leq m',\quad  n\leq n',
\end{equation*}
where,
\begin{equation*}
i:=m-k,\quad i':=m'-k,\quad  j:=n-l,\quad j':=n'-l.
\end{equation*}
The collection of all these structure sheaves with natural homomorphisms $  (\mathcal{G}'(Gr_k^{m'}\times Gr_l^{n'}),\iota^*_{ii'|jj'} ) $ satisfy the inverse limit conditions and determine a sheaf on $ Gr_k^\infty\times Gr_l^\infty$ represented by $ \underset{\leftarrow}{\mathcal{G}} $ .
In fact, every section $ f $ on an arbitrary open set $ U^{\infty}:=(U_{ij})_{i,j} $ in the base space is denoted by a sequence of sections $ (f_{ij})_{i,j} $  satisfying the following relationships: 
\begin{equation*}
f\in \underset{\leftarrow}{\mathcal{G}}(U^{\infty}) \Longleftrightarrow f=(f_{ij})_{i,j},\quad f_{ij}\in \mathcal{G}_{U_{ij}}(U_{ij}),\quad \iota_{ii'|jj'}^*(f_{i'j'})=f_{ij}.
\end{equation*}
According to the structure defined and similar to the normal case, a homomorphism from a finite dimensional  $\nu$-Grassmannian to an infinite dimensional  $\nu$-Grassmannian can be defined as follows:
\begin{equation*}
p_{ij}: {_\nu}Gr_{k|l}^{m|n}\rightarrow     {_\nu}Gr_{k|l}^{\infty},\quad p_{ij}^*((f_{rs}))=f_{ij}.
\end{equation*}
Under this mapping, the structure sheaf $ \mathcal{G}_{ij} $  has a $  \underset{\leftarrow}{\mathcal{G}} $-module structure as well.
\subsection{Super Vector Bundle $ \gamma_{k|l}^{\infty} $}
It is reminded that  $\nu$-Grassmannians are constructed by gluing together  $\nu$-domains $ (V_I,\mathcal{G}_I):=(\mathbb{R}^p,\mathbf{C}^{\infty}_{\mathbb{R}^{p}} \otimes_{\mathbb{R}} \wedge \mathbb{R}^{q}) $ in the direction of homomorphisms  $\varphi_{IJ}$.  Also, the canonical super vector bundle $ \gamma_{k|l}^{m|n}:=(Gr_k^m\times Gr_l^n,\Gamma)  $ over  $ _\nu Gr_{k|l}^{m|n} $ is constructed by gluing the free $ \mathcal{G}_I $-module sheaves $\mathcal{G}_I\otimes<A^I> := \mathcal{G}_I\otimes<A^I_1,...,A^I_{k+l}> $  on $ V_I $ in the direction of the  $ \mathcal{G} $-module homomorphisms:
\begin{align*}
{\psi}_{IJ}^*:\mathcal{G}_J|(V_{JI})\otimes <A^J>|_{JI}   \longrightarrow
\mathcal{G}_I|(V_{IJ})\otimes <A^I>|_{JI},\\
s\otimes<A^J> \mapsto\varphi^*_{IJ}(s)\otimes<(M_J(A^I).id_J)^{-1}.A^I>,
\end{align*}
where $ A_t^I $ is the $ t $-th row of the supermatrix $ A^I $ and $ <A^I_1, ...,A^I_{k+ l}>  $ is the real vector space generated by $ A_t^I $s. It can be seen that each of these free sheaves are equivalent to  $ (V_I,\mathcal{G}_I\otimes\mathbb{R}^{{k|l}}) $.
\\
On the other hand, the canonical homomorphism in  \eqref{a2} induces a canonical homomorphism between  $ \mathcal{G}'_I $-module sheaves in the form of
\begin{align*}
\bar{\iota}^*:\mathcal{G}'_I|({V'_{I})}\otimes <A'^I>&|_{I}\longrightarrow\mathcal{G}_I|({V_{I}})\otimes <A^I>|_{I},\\
&A'^I_t\mapsto A^I_t.
\end{align*}
It can be shown that the following diagram consisting of these homomorphisms are commutative with the gluing homomorphisms:
\begin{equation}
\begin{tabular}{ccc}\label{a4}
$\mathcal{G}'_J|(V'_{JI})\otimes <A'^J>|_{JI}$   &  $\stackrel{\bar{\iota}^*}{\longrightarrow}$  & $\mathcal{G}_J|(V_{JI})\otimes <A^J>|_{JI}$\\
${\psi'}_{IJ}^* \downarrow $   &    &   $\psi_{IJ}^* \downarrow $        \\
$\mathcal{G}'_I|(V'_{IJ})\otimes <A'^I>|_{IJ}$   &   $ \stackrel{\bar{\iota}^*}{\longrightarrow} $  & $\mathcal{G}_I|(V_{IJ})\otimes <A^I>|_{IJ}$.
\end{tabular}
\end{equation}
In this case, if $ U^\infty $  is an open set in $ Gr_k^\infty\times Gr_l^\infty $,  then one can define a canonical homomorphism using the equivalency induced by the above diagrams
\begin{equation*}
{\bar{\iota}}^{*}_{ii'|jj'}|:\Gamma'(V'_I\cap U_{i'j'})\rightarrow \Gamma(V_I\cap U_{ij}),\qquad  m\leq m',\quad  n\leq n',
\end{equation*}
and so
\begin{equation*}
{\bar{\iota}}^{*}_{ii'|jj'}:\Gamma'( U_{i'j'})\rightarrow \Gamma( U_{ij})
\end{equation*}
is defined. The sequence of all these modules and the canonical homomorphisms between them as  $  (\Gamma'( U_{i'j'}),\bar{\iota}^*_{ii'|jj'} ) $ satisfies the inverse limit conditions and induces the module  $ \underset{\leftarrow}{\Gamma}(U^\infty) $. We denote the pre-sheaf  $ U^\infty\mapsto \underset{\leftarrow}{\Gamma}(U^\infty) $ by  $ \underset{\leftarrow}{\Gamma} $ and the corresponding sheaf by  $ \gamma_{k|l}^{\infty} $. This is a sheaf of  $ \underset{\leftarrow}{\mathcal{G}} $-modules that is locally free of rank $ k|l $.
\\
There are  $ \underset{\leftarrow}{\mathcal{G}} $-module homomorphisms from  $ \underset{\leftarrow}{\Gamma}(U^\infty) $ to  $ \Gamma(U_{ij}) $ as:
\begin{equation*}
(s_{rt})\mapsto s_{ij}.
\end{equation*}
These homomorphisms induce the homomorphism  $  \bar{p}_{ij} $ from  $   \gamma_{k|l}^{m|n} $ to  $  \gamma_{k|l}^{\infty} $.
\subsection{Pullback of Super Vector Bundles}
In this section, we show that the super vector bundle  $ \gamma_{k|l}^{\infty} $  is a universal member of the category of super vector bundles. 
\\
With the above assumption and according to Theorem \ref{Sigma}, we have:
\\
Naturally, there is a morphism between the base space of the super vector bundle $\mathcal{E}$ and the base space of the super vector bundle  $ \gamma_{k|l}^{tk|tl} $, for a big enough $ t $, as 
\begin{equation*}
\sigma_t=(\widetilde{\sigma}_t,\sigma^*_t):(M, \mathcal{O})\rightarrow (Gr_{k}^{tk} \times Gr_{l}^{tl}, \mathcal{G}),
\end{equation*}
so that the pullback of  $ \gamma_{k|l}^{tk|tl} $ along  $ \sigma_t  $  is isomorphic to the super vector bundle  $ \cal{E} $. 
\\
$\mathcal{E}$ and the pullback of $\Gamma$ (the canonical super vector bundle over  $_{\nu}Gr$) are isomorphic under  $\sigma$   (Th \ref{sigms iso}). 
Based on what said so far, we have the following theorem: 
\begin{theorem}
	The super vector bundle $ \gamma_{k|l}^{tk|tl} $ and the pullback of $ \gamma_{k|l}^{\infty} $ along $ p_{ij} $ are isomorphic.
\end{theorem}
\begin{proof}
	By definition, the pullback of  	$ \gamma_{k|l}^{\infty} $  along $ p_{ij} $ is in the form of  	$ \mathcal{G} \otimes _{\underset{\leftarrow}{\mathcal{G}}}^{p_{ij} } \underset{\leftarrow}{\Gamma} $ and it suffices to show that the following mapping is an isomorphism: 
	\begin{align*}
	T:\mathcal{G} \otimes _{\underset{\leftarrow}{\mathcal{G}}}^{p_{ij} } \underset{\leftarrow}{\Gamma} \rightarrow \gamma_{k|l}^{tk|tl},\\
	u \otimes (s_{i'j'}) \mapsto u.s_{ij}.
	\end{align*}
	Consider an arbitrary section like
	\begin{equation*}
	s=\sum_{c=1}^{k+l} 1 \otimes (a_{i'j'}^c)(s_{i'j'}^c)
	\end{equation*}
	with a basis like  	$\{ (s_{i'j'}^c)\}_{c=1,...,k+l} $ on a small enough open set like  $ U^\infty $ in  	$ Gr_k^\infty\times Gr_l^\infty $ . The image of this section under $ T $ in $ \Gamma(U) $   is in the form of  	$ \sum_{c=1}^{k+l}a_{ij}^c.s_{ij}^c $ with the basis  	$\{ s_{ij}^c\}_{c=1,...,k+l} $ for $ i=tk-k $ and $ j=tl-l $.
	
	As a result, if  $ T(s) $   is equal to zero, all the coefficients  	$ a_{ij}^c $ are equal to zero and according to the relation 
	\begin{equation*}
	\sum_{c=1}^{k+l} a_{ij}^c \otimes (s_{i'j'}^c)=\sum_{c=1}^{k+l} 1.(a_{i'j'}^c) \otimes (s_{i'j'}^c) =\sum_{c=1}^{k+l} 1 \otimes (a_{i'j'}^c)(s_{i'j'}^c),
	\end{equation*}
	the section $ s $ is zero too; i.e., the kernel of $ T $ is zero and $ T $ is locally an isomorphism. Since the rank of two super vector bundles on a common base space are equal, $ T $ is an isomorphism generally. 
\end{proof}
The following result is directly derived from the last two theorems.
\begin{corollary}\label{sigmat}
	Each super vector bundle  $\mathcal{E}$  with rank $ k|l $ is isomorphic to the pullback of  $ \gamma_{k|l}^{\infty} $  along  $ p_{ij}\circ\sigma_t $, for a big enough $ t $ and $ i=tk-k $ and $ j=tl-l $.
\end{corollary}
There are other morphisms that can replace $ p_{ij}\circ\sigma_t $ in the Corollary \ref{sigmat} we call these morphisms as "desirable morphisms w.r.t. super vector bundle $ \cal{E} $ ". In the next subsection we introduce such a morphism explicitly for an specific example.
\subsection{An example of desirable morphisms}
Consider the infinite dimensional Grassmann manifold $ Gr_k^{\infty} =(Gr_k^{\infty} ,\underset{\leftarrow}{\mathcal{O}}) $ introduced in \ref{Gr-infty}.
We show that the embedding of reduced manifold $ Gr_k^{\infty} $ into the supermanifold $ {_\nu Gr_{k|0}^{\infty}} $ denoted by
\begin{equation*}
\underset{\leftarrow}{\iota}: Gr_k^{\infty}\rightarrow   {_\nu Gr_{k|0}^{\infty}}
\end{equation*}
is a desirable morphism w.r.t. canonical vector bundle over $ Gr_k^{\infty} $. In other words one has the following proposition:
\begin{proposition}
	The pullback of  $ {\gamma_{k|0}^{\infty}} $ along  $ \underset{\leftarrow}{\iota}  $  is isomorphic to the vector bundle  $  \gamma_k^{\infty} $. 
\end{proposition}

\begin{proof}
Let $ \{U_I\} $ be the standard open cover of $ Gr_k^m $. In subsection \ref{grassmannian}, it is shown that the $ \nu $-Grassmannian $ _\nu Gr_{k|0}^{m|n} $ is constructed by gluing $ \nu$-domains $ (V_I,\mathcal{G}_I):=(\mathbb{R}^p,\mathbf{C}^{\infty}_{\mathbb{R}^{p}} \otimes_{\mathbb{R}} \wedge \mathbb{R}^{q}) $. Moreover, $ Gr_k^m $ is constructed by gluing domains $ (U_{I},\mathcal{O}_{I}):=(\mathbb{R}^p,\mathbf{C}^{\infty}_{\mathbb{R}^{p}}) $ through homomorphisms $\phi_{IJ}$, see \eqref{phi}. 
There exists a natural morphism denoted by $ \iota_I^{m, n} $ from $ (U_I, {\mathcal{O}}_I) $ to $ (V_I, {\mathcal{G}}_I) $, for each $ k $ multi-index $ I \subset \{1,2,...,m\}$.
\\
Under ${\iota_I}^{{m, n}^*}:\mathcal{G}_{I}(V_{I})\longrightarrow \tilde{\mathcal{G}}_{I}(V_{I})\cong \mathcal{O}_{I}(U_{I})\label{iI} $, the structure sheaf   $ \mathcal{O}_{I} $  has a $ \mathcal{G}_I $-module structure as well.
The morphisms $ \{\iota_I^{m, n}\}_I $, induce a natural homomorphism $ \iota^{m, n}: Gr_k^m\to\nu Gr_{k|0}^{m|n} $.
Moreover, the collection of all these morphisms i.e. $ \{\iota^{m, n}\}_{m, n} $  determine the homomorphism 
\begin{equation*}\label{iota}
\underset{\leftarrow}{\iota}: Gr_k^{\infty}\rightarrow   {_\nu Gr_{k|0}^{\infty}}.
\end{equation*}
On the other hand, the canonical super vector bundle $ \gamma_{k|0}^{m|n}:=(Gr_k^m\times \{*\},\Gamma)  $ over  $ _\nu Gr_{k|0}^{m|n} $ is constructed by gluing the free $ \mathcal{G}_I $-module sheaves $\mathcal{G}_I\otimes<A^I> := \mathcal{G}_I\otimes<A^I_1,...,A^I_{k}> $  on $ V_I $ through the  $ \mathcal{G} $-module homomorphisms:
\begin{align*}
{\psi}_{IJ}^*:\mathcal{G}_J|(V_{JI})\otimes <A^J>|_{JI}   \longrightarrow
\mathcal{G}_I|(V_{IJ})\otimes <A^I>|_{JI}\\
s\otimes<A^J> \mapsto\varphi^*_{IJ}(s)\otimes<(M_J(A^I).id_J)^{-1}.A^I>,
\end{align*}
where $ \varphi_{IJ}=(id, \varphi_{IJ}^*) $ is gluing morphism defined in step 3 of construction of the $ \nu$-Grassmannians in page 6, and $ A_t^I $ is the $ t $-th row of the supermatrix $ A^I $. In addition, $ <A^I_1, ...,A^I_{k}>  $ is a real vector space generated by $ A_t^I $s. It can be seen that each of these free sheaves are equivalent to  $ (V_I,\mathcal{G}_I\otimes\mathbb{R}^{{k|0}}) $.
\\
The homomorphism $ {\iota_I}^{{m, n}^*} $, induces a canonical homomorphism between sheaves of $ \mathcal{G}_I $-modules as follows 
\begin{align*}
\bar{\iota}_I^{{m, n}^*}:\mathcal{G}_I|({V_{I})}\otimes <A^I>&|_{I}\longrightarrow\mathcal{O}_I|({U_{I}})\otimes <\tilde{A}^I>|_{I},\\
&A^I_t\mapsto \tilde{A}^I_t.
\end{align*}
The collection of all these homomorphisms determine a homomorphism as follows
\begin{equation*}
\underset{\leftarrow}{\bar{\iota}}: \gamma_k^{\infty}\rightarrow   {\gamma_{k|0}^{\infty}}.
\end{equation*}
By definition, the pullback of $ {\gamma_{k|0}^{\infty}} $ along $ \underset{\leftarrow}{\iota} $ is in the form of $ \underset{\leftarrow}{\mathcal{O}} \otimes _{\underset{\leftarrow}{\mathcal{G}}}^{\underset{\leftarrow}{\iota}} \underset{\leftarrow}{\Gamma} $ and it can be easily shown that the following homomorphism is an isomorphism: 
\begin{align*}
T:\underset{\leftarrow}{\mathcal{O}} \otimes _{\underset{\leftarrow}{\mathcal{G}}}^{\underset{\leftarrow}{\iota}} \underset{\leftarrow}{\Gamma} &\rightarrow \underset{\leftarrow}{\tilde{\Gamma}},\qquad \tilde{s}_{i}=\underset{\leftarrow}{{\bar{\iota}}^*}(s_{i}),\\
(u_{i}) \otimes (s_{i}) &\mapsto (u_{i}.\tilde{s}_{i}).
\end{align*}
This completes the proof.
\end{proof}

\end{document}